\documentclass[12pt]{article}
\usepackage{comment}
\usepackage{amsfonts}
\usepackage{amsmath}
\usepackage{amssymb}
\usepackage{amsthm} 
\usepackage{enumerate}
\usepackage[colorlinks=true, allcolors=blue]{hyperref}
\usepackage{tikz}

\theoremstyle{plain}
\newtheorem{theorem}{Theorem}
\newtheorem{lemma}[theorem]{Lemma}
\newtheorem{corollary}[theorem]{Corollary}

\newcommand\1{\mathbf{1}}
\newcommand\tr{\mathrm{tr}}

\newcommand\comp[1]{{\mkern2mu\overline{\mkern-2mu#1}}}

\begin{document}
%\title{First-order logic characterizations of  distance-regularized  and controllable graphs}
\title{ Descriptive complexity of \\ controllable graphs}
\author{
Aida Abiad
\thanks{\texttt{a.abiad.monge@tue.nl},  Department of Mathematics and Computer Science, Eindhoven University of Technology, The Netherlands}
\thanks{Department of Mathematics: Analysis, Logic and Discrete Mathematics, Ghent University, Belgium} \thanks{ Department of Mathematics and Data Science, Vrije Universiteit Brussel, Belgium} 
\and
Anuj Dawar 
\thanks{\texttt{anuj.dawar@cl.cam.ac.uk}, Department of Computer Science and Technology, University of Cambridge, UK}
\and 
Octavio Zapata
\thanks{\texttt{octavio@im.unam.mx},  Instituto de Matem\'aticas, Universidad Nacional Aut\'onoma de M\'exico, M\'exico}
} 
\date{}
\maketitle
\begin{abstract}
Let $G$ be a graph on $n$ vertices with adjacency matrix $A$, and let $\1$ be the all-ones vector. 
We call $G$ \emph{controllable} if the set of vectors $\1, A\1, \dots, A^{n-1}\1$ spans the whole space $\mathbb{R}^n$. 
We characterize the isomorphism problem of controllable graphs in terms of other combinatorial, geometric and logical problems. 
We also describe a polynomial time algorithm for graph isomorphism that works for almost all graphs. 
\end{abstract}
 
%%%%%%%%%%%%%%%%%%%%%%%%%%%%%%%%%%%%%%%%%%%%%%%%%%%%%%%%%%%%%%%%%%%%%%%%%%%%%%%%%%%%%%%%%%%%%%%
\section{Introduction}
%%%%%%%%%%%%%%%%%%%%%%%%%%%%%%%%%%%%%%%%%%%%%%%%%%%%%%%%%%%%%%%%%%%%%%%%%%%%%%%%%%%%%%%%%%%%%%%
One of the most important open questions in spectral graph theory is to determine to what extent are graphs characterized by their spectrum (see e.g. \cite{van2003graphs, van2009developments}). 
The spectrum of a finite simple graph with $n$ vertices is the sequence of $n$ eigenvalues of its adjacency matrix, counting multiplicities.  
We say that a graph $G$ is \emph{determined by its spectrum} if the spectrum of $G$ is different from the spectrum of any other graph which is not isomorphic to $G$.  
For example, the complete graph $K_n$, the cycle $C_n$, and the path $P_n$ are graphs determined by their spectrum.
On the other hand, most trees \cite{schwenk1973almost} and strongly regular graphs  \cite{fon2002new} are examples of graphs that are not determined by their spectrum.
In fact given a graph $G$, there are criteria that allow us to construct a new graph with the same spectrum of $G$ but not isomorphic to $G$ (see e.g. \cite{godsil1982constructing}).
In contrast with this, it has been observed that randomly generated graphs tend to be determined by their spectrum and the spectrum of their complement   \cite{wang2006sufficient}.

It is clear that in general the spectrum is not sufficient to characterize a graph, but we would like to know the asymptotic behavior of the number of graphs determined by its spectrum.
Are they the majority or just a few?
What happens if, in addition to the spectrum of a graph, we consider the spectrum of its complement?
Can we find classes of graphs determined by their spectrum with non-trivial combinatorial properties?
It has been conjectured that the proportion of graphs on $n$ vertices which are determined by the spectrum and the spectrum of its complement goes to 1 as $n$ tends to infinity. 
Wang et al.~\cite{W2013,WX,MLW,W2017} have a number of results supporting this conjecture. 
They gave sufficient conditions for a graph to be determined by the spectrum and the spectrum of its complement. 
The majority of their results are proven for a wide class of graphs, the so-called controllable graphs. 
This class was introduced explicitly by Godsil and Severini \cite{godsil2010control} in their study of quantum walks on graphs. 

Let $G$ be a finite simple graph on $n$ vertices with adjacency matrix $A$.  
If we write  $\1$ for the vector with all entries equal to 1, then the \emph{walk matrix} of $G$ is, by definition, the $n\times n$ matrix% 
\[
    W_G =  \big[ \1\quad A\1\quad \cdots \quad A^{n-1}\1 \big].
\] 
The $ij$-entry of the walk matrix $W_G$ counts the number of walks in $G$ of length $j-1$ starting at vertex $i$.  
 We say that the graph $G$ is \emph{controllable} if its walk matrix $W_{G}$ is invertible. 
If $G$ is regular of degree $k$, then $A\1 = k\1$; this implies that $W_G$ has rank 1. 
It follows that controllable graphs cannot be regular. 
We note also that if $P$ is a permutation matrix that commutes with $A$, then $PA^r\1 = A^r P\1 = A^r\1$ for all $r=0,\dots,n-1$, and hence $PW_G = W_G$; this implies that $P=I$ when $W_G$ is invertible.  
Therefore, the only automorphism of a controllable graph is the identity. 
The theory of controllable graphs was developed by Godsil in \cite{godsil2012controllable}, where it was conjectured that the proportion of graphs on $n$ vertices which are controllable goes to 1 as $n \to \infty$. 
It was later confirmed by O'Rourke and  Touri \cite{o2016conjecture} that indeed almost all graphs are controllable. 

%%%%%%%%%%%%%%%%%%%%%%%%%%%%%%%%%%%%%%%%%%%%%%%%%%%%%%%%%%%%%%%%%%%%%%%%%%%%%%%%%%%%%%%%%%%%%%%
\section{Generalized cospectrality and\\ walk-equivalence}
%%%%%%%%%%%%%%%%%%%%%%%%%%%%%%%%%%%%%%%%%%%%%%%%%%%%%%%%%%%%%%%%%%%%%%%%%%%%%%%%%%%%%%%%%%%%%%%
The \emph{characteristic polynomial}  of a graph $G$ on $n$ vertices with adjacency matrix $A$ is, by definition, the polynomial $\det(tI - A)$ (i.e., it is the characteristic polynomial of $A$).
We say that two graphs are \emph{cospectral} if they have the same characteristic polynomial. 
Since isomorphic graphs have permutation similar adjacency matrices, it follows that isomorphic graphs are cospectral.  
If we write $J = \1\1^T$ for the all-ones matrix, then the polynomial $\det(tI - sJ - A)$ is called the \emph{generalized characteristic polynomial} of $G$. 
Two graphs are called \emph{generalized cospectral} if they have the same polynomial $\det(tI - sJ - A)$ for all values of $s$. 
Since $\det(tI - sJ - A)$ is the characteristic polynomial of the matrix $A + sJ$, it follows that isomorphic graphs are generalized cospectral. 
Since $\det(tI - sJ - A) = \det(tI - A)$ when $s = 0$, we have that generalized cospectral graphs are cospectral. 
If we write $\comp{A}$ for the adjacency matrix of the complement of $G$, then we have that $\comp{A}=J-I-A$, and hence that $\det(tI- \comp{A}) = (-1)^n\det((-t-1)I - (-1) J - A)$.
Therefore, having cospectral complements is a necessary condition for being generalized cospectral. 
An important result of Johnson and Newman \cite{JN} says that this condition is also sufficient: cospectral graphs with cospectral complements are generalized cospectral. 
The smallest example of two non-isomorphic graphs that are generalized cospectral is shown in Figure \ref{fig:smallestRcospectralpair}. 

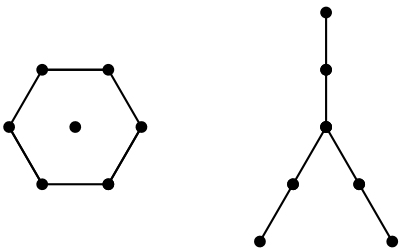
\begin{figure}
\begin{center}
\begin{tikzpicture}[x=0.55pt,y=0.55pt,yscale=-1,xscale=1, thick]
%Shape: Regular Polygon [id:dp9554121248363912] 
\draw   (123,100.5) -- (100.25,139.9) -- (54.75,139.9) -- (32,100.5) -- (54.75,61.1) -- (100.25,61.1) -- cycle ;
%Straight Lines [id:da9025500968364994] 
\draw    (54.75,61.1) -- (100.25,61.1) ;
\draw [shift={(100.25,61.1)}, rotate = 0] [color={rgb, 255:red, 0; green, 0; blue, 0 }  ][fill={rgb, 255:red, 0; green, 0; blue, 0 }  ][line width=0.75]      (0, 0) circle [x radius= 3.35, y radius= 3.35]   ;
\draw [shift={(54.75,61.1)}, rotate = 0] [color={rgb, 255:red, 0; green, 0; blue, 0 }  ][fill={rgb, 255:red, 0; green, 0; blue, 0 }  ][line width=0.75]      (0, 0) circle [x radius= 3.35, y radius= 3.35]   ;
%Straight Lines [id:da5303127264396368] 
\draw    (123,100.5) -- (100.25,139.9) ;
\draw [shift={(100.25,139.9)}, rotate = 120] [color={rgb, 255:red, 0; green, 0; blue, 0 }  ][fill={rgb, 255:red, 0; green, 0; blue, 0 }  ][line width=0.75]      (0, 0) circle [x radius= 3.35, y radius= 3.35]   ;
\draw [shift={(123,100.5)}, rotate = 120] [color={rgb, 255:red, 0; green, 0; blue, 0 }  ][fill={rgb, 255:red, 0; green, 0; blue, 0 }  ][line width=0.75]      (0, 0) circle [x radius= 3.35, y radius= 3.35]   ;
%Straight Lines [id:da5285589580899261] 
\draw    (32,100.5) -- (54.75,139.9) ;
\draw [shift={(54.75,139.9)}, rotate = 60] [color={rgb, 255:red, 0; green, 0; blue, 0 }  ][fill={rgb, 255:red, 0; green, 0; blue, 0 }  ][line width=0.75]      (0, 0) circle [x radius= 3.35, y radius= 3.35]   ;
\draw [shift={(32,100.5)}, rotate = 60] [color={rgb, 255:red, 0; green, 0; blue, 0 }  ][fill={rgb, 255:red, 0; green, 0; blue, 0 }  ][line width=0.75]      (0, 0) circle [x radius= 3.35, y radius= 3.35]   ;
%Straight Lines [id:da5880907280373837] 
\draw  [dash pattern={on 0.75pt off 750pt}]  (77.5,100.5) -- (100.25,139.9) ;
\draw [shift={(100.25,139.9)}, rotate = 60] [color={rgb, 255:red, 0; green, 0; blue, 0 }  ][fill={rgb, 255:red, 0; green, 0; blue, 0 }  ][line width=0.75]      (0, 0) circle [x radius= 3.35, y radius= 3.35]   ;
\draw [shift={(77.5,100.5)}, rotate = 60] [color={rgb, 255:red, 0; green, 0; blue, 0 }  ][fill={rgb, 255:red, 0; green, 0; blue, 0 }  ][line width=0.75]      (0, 0) circle [x radius= 3.35, y radius= 3.35]   ;

%Straight Lines [id:da7018125623815556] 
\draw    (250,100.5) -- (272.75,139.9) ;
\draw [shift={(272.75,139.9)}, rotate = 60] [color={rgb, 255:red, 0; green, 0; blue, 0 }  ][fill={rgb, 255:red, 0; green, 0; blue, 0 }  ][line width=0.75]      (0, 0) circle [x radius= 3.35, y radius= 3.35]   ;
\draw [shift={(250,100.5)}, rotate = 60] [color={rgb, 255:red, 0; green, 0; blue, 0 }  ][fill={rgb, 255:red, 0; green, 0; blue, 0 }  ][line width=0.75]      (0, 0) circle [x radius= 3.35, y radius= 3.35]   ;
%Straight Lines [id:da6370426310296347] 
\draw    (250,100.5) -- (227.25,139.9) ;
\draw [shift={(227.25,139.9)}, rotate = 120] [color={rgb, 255:red, 0; green, 0; blue, 0 }  ][fill={rgb, 255:red, 0; green, 0; blue, 0 }  ][line width=0.75]      (0, 0) circle [x radius= 3.35, y radius= 3.35]   ;
\draw [shift={(250,100.5)}, rotate = 120] [color={rgb, 255:red, 0; green, 0; blue, 0 }  ][fill={rgb, 255:red, 0; green, 0; blue, 0 }  ][line width=0.75]      (0, 0) circle [x radius= 3.35, y radius= 3.35]   ;
%Straight Lines [id:da02803171778552005] 
\draw    (227.25,139.9) -- (204.5,179.31) ;
\draw [shift={(204.5,179.31)}, rotate = 120] [color={rgb, 255:red, 0; green, 0; blue, 0 }  ][fill={rgb, 255:red, 0; green, 0; blue, 0 }  ][line width=0.75]      (0, 0) circle [x radius= 3.35, y radius= 3.35]   ;
\draw [shift={(227.25,139.9)}, rotate = 120] [color={rgb, 255:red, 0; green, 0; blue, 0 }  ][fill={rgb, 255:red, 0; green, 0; blue, 0 }  ][line width=0.75]      (0, 0) circle [x radius= 3.35, y radius= 3.35]   ;
%Straight Lines [id:da7539549345297953] 
\draw    (272.75,139.9) -- (295.5,179.31) ;
\draw [shift={(295.5,179.31)}, rotate = 60] [color={rgb, 255:red, 0; green, 0; blue, 0 }  ][fill={rgb, 255:red, 0; green, 0; blue, 0 }  ][line width=0.75]      (0, 0) circle [x radius= 3.35, y radius= 3.35]   ;
\draw [shift={(272.75,139.9)}, rotate = 60] [color={rgb, 255:red, 0; green, 0; blue, 0 }  ][fill={rgb, 255:red, 0; green, 0; blue, 0 }  ][line width=0.75]      (0, 0) circle [x radius= 3.35, y radius= 3.35]   ;
%Straight Lines [id:da2668154846765045] 
\draw    (250,61.1) -- (250,100.5) ;
\draw [shift={(250,100.5)}, rotate = 90] [color={rgb, 255:red, 0; green, 0; blue, 0 }  ][fill={rgb, 255:red, 0; green, 0; blue, 0 }  ][line width=0.75]      (0, 0) circle [x radius= 3.35, y radius= 3.35]   ;
\draw [shift={(250,61.1)}, rotate = 90] [color={rgb, 255:red, 0; green, 0; blue, 0 }  ][fill={rgb, 255:red, 0; green, 0; blue, 0 }  ][line width=0.75]      (0, 0) circle [x radius= 3.35, y radius= 3.35]   ;
%Straight Lines [id:da735689428218983] 
\draw    (250,21.69) -- (250,61.1) ;
\draw [shift={(250,61.1)}, rotate = 90] [color={rgb, 255:red, 0; green, 0; blue, 0 }  ][fill={rgb, 255:red, 0; green, 0; blue, 0 }  ][line width=0.75]      (0, 0) circle [x radius= 3.35, y radius= 3.35]   ;
\draw [shift={(250,21.69)}, rotate = 90] [color={rgb, 255:red, 0; green, 0; blue, 0 }  ][fill={rgb, 255:red, 0; green, 0; blue, 0 }  ][line width=0.75]      (0, 0) circle [x radius= 3.35, y radius= 3.35]   ;
\end{tikzpicture}
\end{center}
\caption{Smallest pair of non-isomorphic generalized cospectral graphs with respect to the adjacency matrix.}\label{fig:smallestRcospectralpair}
\end{figure}

We recall that the $ij$-entry of $A^r$ is equal to the number of walks in $G$ of length $r\geq 0$ from vertex $i$ to vertex $j$. 
In particular, there is a walk of length zero from each vertex to itself because $A^0 = I$. 
It follows from basic properties of formal power series (see e.g. \cite[p. 40]{godsil1993algebraic}) that
\[
     \sum_{r \geq 0} A^r t^r = (I- tA)^{-1}. 
\]
Since the total number of walks in $G$ of length $r$ is equal to 
\[\tr(A^r J) = \1^T A^r \1,\] 
the generating function for all walks in $G$ is given by 
\[
    \sum_{r \geq 0}\tr(A^r J) t^r = \1^T(I- tA)^{-1} \1. 
\]
We say that two graphs are \emph{walk-equivalent} if their generating functions for all walks are equal. 
Note that for every real number $t$, we have  
\[
  (t-1)I - \bar{A} =  tI - (J -A) = (tI + A)(I-(tI+A)^{-1} J).
\]
Since \[I-(tI+A)^{-1} J = I-((tI+A)^{-1}\1 )\1^T\] and 
\[(t-1)I - \bar{A} =   ((-1)(-tI - A))(I-((tI+A)^{-1}\1 )\1^T),
\] 
we can use the identity $\det(I-\mathbf{u}\mathbf{v}^T) = 1 - \mathbf{v}^T\mathbf{u}$ to find that
\[
     \frac{\det((t-1)I-\comp{A})}{(-1)^{n}\det (-tI - A) } = 1 - \1^T(tI + A)^{-1} \1. 
\]
Consequently, we have that
\[
      \1^T(I - tA)^{-1} \1 = \frac{1}{t} \left( \frac{\det((-t^{-1}-1)I - \comp{A})}{(-1)^{n}\det( t^{-1}I- A)} - 1 \right). 
\]
Therefore, the generating function for all walks is determined by the characteristic polynomial of a graph and the characteristic polynomial of its complement.
Hence, a necessary condition for generalized cospectrality is walk-equivalence.
It follows from \cite[Corollary 3.2]{godsil2012controllable} that, for controllable graphs, this condition is also sufficient. 

\begin{theorem}[\cite{godsil2012controllable}]\label{lem:sufficient}
Two controllable graphs are walk-equivalent if and only if they are generalized cospectral. 
\end{theorem}

%%%%%%%%%%%%%%%%%%%%%%%%%%%%%%%%%%%%%%%%%%%%%%%%%%%%%%%%%%%%%%%%%%%%%%%%%%%%%%%%%%%%%%%%%%%%%%%
\section{First-order logic with counting quantifiers}
%%%%%%%%%%%%%%%%%%%%%%%%%%%%%%%%%%%%%%%%%%%%%%%%%%%%%%%%%%%%%%%%%%%%%%%%%%%%%%%%%%%%%%%%%%%%%%%
Descriptive complexity is the subfield of mathematical logic that studies the formal relationship between logical complexity and algorithmic efficiency. 
There are efficient algorithms that determine whether two graphs satisfy exactly the same properties if we consider only properties that can be described in first-order logic using finitely many variables.
%if all the properties under consideration can be described in first-order logic using only finitely many variables. 
%for logics with finitely many variables. 
%For the full first-order logic, no such algorithm is known because this is equivalent to the graph isomorphism problem. 
The first-order logic of graphs consists of strings of symbols built using  variables ($x,y,z,\dots$), the usual logical connectives for negation ($\lnot$) and  for disjunction ($\lor$),  the existential quantifier ($\exists$) , various types of parentheses used to avoid ambiguity, and the binary relation symbols  for equality ($=$) and  for adjacency ($E$).
The variables that occur in expressions formed using these symbols always range over the vertices of a graph. 
As a consequence, the quantifiers only apply to individual vertices and this is why the logic is called \emph{first-order}.

There are certain rules  within  the language of first-order logic regarding the formation of  interpretable expressions, also known as \emph{formulas}. 
For example, the expressions
\[
Exy, \qquad\qquad  \lnot\exists x \lnot Exx \qquad\qquad \textnormal{ and } \qquad \qquad \exists x\exists y\lnot(x = y \lor  Exy)
\]
are first-order formulas of the language of graphs. 
The way to interpret the formula $Exy$ in a given graph $G$ goes as follows.  
First we choose two vertices of $G$, say $u$ and $v$, and substitute them for the variables $x$ and $y$ to obtain the expression $Euv$. 
Next we verify if there is an edge in $G$ between $u$ and $v$. 
If there is indeed such an edge, we say that the formula $Exy$ is \emph{true} in $G$, or that $G$ \emph{satisfies} $Exy$, when we interpret the variable $x$ as vertex $u$ and the variable $y$ as vertex $v$. 
We denote this by writing $G,u,v \models Exy$ or, equivalently, by $G \models Euv$. 
If there is no such edge in $G$, then the formula is not true in the graph for that particular choice of assignment of vertices to variables.  
This means that the formula $Exy$ asserts the existence of an edge in the graph where we interpret it.  

It is customary to introduce other usual symbols, connectives and quantifiers as abbreviations. 
For example, the formula $\lnot(x = y \lor  Exy)$ can be rewritten as $(x\neq y \land \lnot Exy)$ and the formula $\lnot\exists x \lnot Exx$ as $\forall x Exx$.
The unquantified variables of a formula are called \emph{free variables}. 
For example, the two variables $x$ and $y$  that occur inside the formula $\phi$ defined by
\[
\phi  := \exists x\exists y(x\neq y \land \lnot Exy)
\] 
are within the scope of some quantifier. 
This implies that the number of free variables in $\phi$ is zero. 
A \emph{sentence} is a formula that does not contain any free  variable. 
Note that the sentence $\phi$ defined above is true in a graph $G$ if and only if there are (at least) two distinct non-adjacent vertices in $G$. 

Two measures of logical complexity for a formula $\phi$ are  the maximum number of free variables in any subformula of $\phi$ and the depth of nesting of the quantifiers in $\phi$. 
We say that a sentence $\phi$ \emph{distinguishes} a graph $G$ from a graph $H$ if  $\phi$ is true in one graph and not true in the other, i.e., if  $G\models\phi$ and $H\not\models\phi$ or viceversa. 
If $\phi$ distinguishes $G$ from any non-isomorphic graph $H$, then we say that $\phi$ \emph{defines} $G$ (up to isomorphism). 
Every finite graph $G$ is definable by a canonical first-order sentence $\phi_G$ (see e.g. \cite[Lemma 3.4]{libkin2004elements}). 
If the number of vertices in $G$ is $n$, then the number of distinct variables used in its defining sentence $\phi_G$ is  $n+1$.
Since there are efficient isomorphism tests for classes of graphs defined by sentences with low logical complexity, it is a relevant task to find short definitions for interesting classes. 

The language $L^k$ consists of the fragment of first-order logic where the formulas are restricted to use at most $k\geq 1$ distinct variables.  
We use $C^k$ to denote extension of $L^k$ with \emph{counting quantifiers}: for each non-negative integer $d$, there is a quantifier $\exists^{\geq d}$ whose semantics is defined so that $\exists^{\geq d}x \phi$ is true in a graph $G$ if there are at least $d$ distinct vertices of $G$ that can be substituted for $x$ to make $\phi$ true. 
We use the abbreviation $\exists^{d}\phi$ for the formula $\exists^{\geq d}\phi \land \lnot \exists^{\geq d+1}\phi$ that asserts the existence of exactly $d$ vertices satisfying $\phi$. 
For example, the sentence $\forall x \exists^{d} y  Exy$ of the language $C^2$ is true in a graph $G$ if and only if $G$ is regular of degree $d$. 
Consequently, any two regular graphs of different degree can be distinguished by a $C^2$-sentence.  

Two graphs $G$ and $H$ are \emph{elementary equivalent} with respect to a first-order language $L$ (or \emph{$L$-equivalent}), just in case $G \models \phi$ if and only if $H \models \phi$ for any $L$-sentence $\phi$. 
In other words, $L$-equivalent graphs are precisely those graphs that cannot be distinguished by any property definable by a sentence of the language $L$. 
There is an algorithm named after  Weisfeiler and Leman that, for every $k\geq 1$, determines in polynomial time  whether two graphs are $C^k$-equivalent (see e.g. \cite{immerman1990describing}). 
It is well-known that if two graphs are $C^3$-equivalent, then they are generalized cospectral (see e.g. \cite{alzaga2010spectra, DSZ, furer2010power, rattan2023weisfeiler}). 
The converse is false; the two graphs of Figure \ref{fig:smallestRcospectralpair} are distinguishable by the sentence $\exists x \forall y \lnot Exy$, which asserts the existence of an isolated vertex. 
The use of counting is essential since, for every $k$, there is a pair of non-isomorphic $L^k$-equivalent graphs which are not generalized cospectral  (see \cite[Proposition 4]{DSZ}).  
Also, the use of three variables is necessary because if we let $G$ be the disjoint union of two triangles and let $H$ be a cycle of length 6, then it can be shown that $G$ and $H$ are $C^2$-equivalent, but $G$ and $H$ are not cospectral. 
 
 %%%%%%%%%%%%%%%%%%%%%%%%%%%%%%%%%%%%%%%%%%%%%%%%%%%%%%%%%%%%%%%%%%%%%%%%%%%%%%%%%%%%%%%%%%%%%%%
\section{Isomorphism approximations}
%%%%%%%%%%%%%%%%%%%%%%%%%%%%%%%%%%%%%%%%%%%%%%%%%%%%%%%%%%%%%%%%%%%%%%%%%%%%%%%%%%%%%%%%%%%%%%%
We now proceed to describe the relation of $C^2$-equivalence to other combinatorial and geometric approximations of graph isomorphism. 
Recall that the degree of a vertex $v$ in a graph $G$ is, by definition, the number of vertices in $G$ which are adjacent to $v$; it is denoted by $d(v)$. 
The \emph{degree sequence} of $G$ is the integer sequence $d(G)$ defined by $d(G)=\{d(v):v\in V(G)\}$. 
If we write $N(v)$ for the set of all those vertices in $G$ that are adjacent to $v$, then the sequence $\{d_r(v):r\geq 0\}$ is defined inductively by $d_0(v)=d(v)$ and $d_{r+1}(v)=\{d_r(u):u\in N(v)\}$ for every $r$.
The \emph{iterated degree sequence} of $G$ is the sequence $D(G)=\{d_r(G):r\geq 0\}$ defined inductively by $d_0(G)=d(G)$ and $d_{r+1}(G)=\{d_r(v):v\in V(G)\}$ for every $r$.
The process of finding the iterated degree sequence of a graph has several widely adopted names; it is known as canonical labelling, color refinement, naive vertex classification or 1-dimensional Weisfeiler-Leman algorithm.

Indistinguishability by iterated degree sequences is a strong isomorphism invariant in the sense that it works for almost all graphs. 
Indeed, a classical result of Babai, Erd\H{o}s and Selkow \cite{babai1980random} says that if $G$ is a random graph on $n$ vertices with edge probability $1/2$, then every graph with the same iterated degree sequence of $G$ is isomorphic to $G$ asymptotically almost surely. 
However, indistinguishability by iterated degree sequences is weak in the sense that two regular graphs with the same number of vertices and the same degree necessarily have the same iterated degree sequence. 
For example, if $G$ is the disjoint union of two triangles and $H$ is the cycle of length 6, then both $G$ and $H$ have $6$ vertices and degree $2$, and hence their iterated degree sequence is $(2,2,2,2,2,2)$.   
It is well-known that a necessary and sufficient condition for indistinguishability by iterated degree sequences is $C^2$-equivalence (see \cite[Theorem 4.8.1]{immerman1990describing}).

\begin{theorem}[\cite{immerman1990describing}]
	Two graphs have the same iterated degree sequence if and only if they are $C^2$-equivalent. 
\end{theorem}

It turns out that the combinatorial notion of having the same iterated degree sequence is equivalent to a  the geometric notion called fractional isomorphism (see \cite[Theorem 2.2]{ramana1994fractional}). 
A real matrix $S$ is called \emph{doubly stochastic} if all its entries are non-negative and every row and every column sums to 1.  
The Birkhoff-von Neumann theorem says that the set of all $n\times n$ doubly stochastic matrices is a compact and convex set whose extreme points are the permutation matrices (see e.g. \cite[Theorem 8.2.2]{horn2012matrix}). 
Recall that two graphs $G$ and $H$ with adjacency matrices $A$ and $B$ are isomorphic if and only if there exists a permutation matrix $P$ such that $PAP^T = B$. 
If we multiply both sides by $P$, then we get the equivalent condition $PA=BP$. 
We say that the graphs $G$ and $H$ are \emph{fractionally isomorphic} if there exists a doubly stochastic matrix $S$ such that $SA = BS$.   

\begin{theorem}[\cite{ramana1994fractional}]\label{thm:degree}
	Two graphs are fractionally isomorphic if and only if they have the same iterated degree sequence. 
\end{theorem}

We turn now to investigate the connection between $C^2$-equivalence (or, equivalently, indistinguishability by iterated degree sequences, or fractional isomorphism) and the notion of walk-equivalence. 

\begin{lemma} \label{lem:c2}
   If two graphs are $C^2$-equivalent, then they are walk-equivalent.
\end{lemma}   

\begin{proof} 
We shall write a formula $\psi^q_{r}(x)$ of first-order logic with counting such that if $G$ is a graph and $u$ is a vertex of $G$, then $G\models \psi^q_{r}(u)$ if and only if there are $q\geq 0$ walks  in $G$ of length $r \geq 0$ starting at $u$.
We proceed to define $\psi^q_{r}(x)$ by induction on $r$. 
If $r = 0$, then we define 
\[
\psi^0_{0}(x) := \bot, \qquad \psi^1_{0}(x) := \top
 \qquad \textnormal{and} \qquad 
\psi^q_{0}(x) := \bot \quad \textnormal{for}\ q>1,
\]
where $\bot$ represents any false formula, e.g. $\forall y(Exy \land \lnot Exy)$, and $\top$ represents any tautology, e.g. $\forall y(Exy \lor \lnot Exy)$. 
Now if $r = 1$, then we define
\[
\psi^0_{1}(x) := \forall y\lnot Exy \qquad \textnormal{and} \qquad 
\psi^q_{1}(x) := \exists^{r} y\  Exy \quad \textnormal{for}\ q>0.
\]
For $r > 1$, we define
\[
\psi^0_{r+1}(x) := \forall y ( Exy \to \psi^0_{r}(y) ),
\]
and if $r >0$, then we define
\[
\psi^q_{r+1}(x) := \bigvee_{(q_1^{a_1},\dots,q_d^{a_d})\in \Pi_q} [ ( \bigwedge_{i = 1}^d \exists^{a_i}y\ \psi_{r}^{q_i}(y)) \land \exists^{\geq a}y\ Exy],
\]
where $\Pi_q$ denote the set of all integer partitions of $q$ (i.e., $q_i\geq 0$, $a_i\geq 1$ and $q= \sum_{i=1}^d a_iq_i$), and $a = \sum_{i=1}^d a_i$. 
We observe that in all these definitions we do not use more than two distinct variables.

Having defined the formula $\psi^r_{r}(x)$, using the same notation we define the sentence $\phi_{r}^{q}$ as follows:
\[
    \phi_{r}^{q} := \bigvee_{(q_1^{a_1},\dots,q_d^{a_d})\in \Pi_q}\  \bigwedge_{i = 1}^d \exists^{a_i}x\ \psi_{r}^{q_i}(x).
\]
By definition, we have that $G\models  \phi_{r}^{q}$ if and only if there are $r$ walks in $G$ of length $r$. 

Finally, suppose that $G$ and $H$ are two graphs which are not walk-equivalent. 
This implies that $G$ and $H$ have a different number of walks of length $r$ for some $r\geq 0$. 
If we write $q$ for the number of walks in $G$ of length $r$, then we have that $G\models  \phi_{r}^{q}$ and $H \not\models  \phi_{r}^{q}$. 
Since $\phi_r^q$ is a sentence of the counting logic $C^2$, it follows that $G$ and  $H$ are not $C^2$-equivalent. 
\end{proof}

A necessary condition for generalized cospectrality of two graphs $G$ and $H$ is that their walk matrices satisfy $W^T_{G} W_{G} = W^T_{H} W_{H}$ (see e.g. \cite[Lemma 3]{van2009developments}). 
Our next result implies that this condition on the walk matrices is also necessary for $C^2$-equivalence. 

\begin{lemma} \label{lem:walk}
    If the graphs $G$ and $H$ are $C^2$-equivalent, then there exists a permutation matrix $P$ such that $PW_G = W_H$.
\end{lemma}   

\begin{proof} 
We shall use the $C^2$-formulas $\psi^{q}_{r}(x)$ already defined in the proof of Lemma \ref{lem:c2}. 
Recall that $\psi^{q}_{r}(x)$ asserts the existence of exactly $q$ walks of length $r$ starting at vertex $x$.
Since $G$ and $H$ are $C^2$-equivalent, 
there is a vertex $v$ of $G$ such that $G\models \psi^{q}_{r}(v)$ if and only if there is a vertex $v'$ of $H$ such that $H \models \psi^{q}_{r}(v')$. 
Hence the mapping $v \mapsto v'$ is a bijection between the sets $\{v \in V(G): G\models\psi^{q}_{r}(v)\}$ and $\{v' \in V(H): H \models\psi^{q}_{r}(v')\}$ for each $q\geq 0$ and $r \geq 1$. 
Since the rows of the walk matrix of a graph are indexed by the vertices of the graph, it follows that  the above bijection determines is a permutation matrix $P$ such that $PW_{G} = W_{H}$. 
\end{proof}

It follows from \cite[Lemma 6.1]{godsil2012controllable} that if two graphs $G$ and $H$ with adjacency matrices $A$ and $B$ are generalized cospectral and controllable, then the matrix $Q = W_{H}W_{G}^{-1}$ satisfies $QAQ^T= B$ and $Q\1 = \1$.
We shall use this remark to prove our next result. 

 \begin{theorem} \label{thm:iso}
     Two controllable graphs are isomorphic if and only if they are $C^2$-equivalent.
 \end{theorem}
 
\begin{proof} 
We prove that $C^2$-equivalent controllable graphs are isomorphic. 
The converse is trivially true; isomorphic graphs are $C^2$-equivalent, irrespectively if they are controllable or not. 
Consider two controllable graphs $G$ and $H$ with adjacency matrices $A$ and $B$, respectively. 
If $G$ and $H$ are  $C^2$-equivalent, then it follows from Lemma \ref{lem:c2} that  $G$ and $H$ are walk-equivalent.
Thus, from Theorem \ref{lem:sufficient}, we know that $G$ and $H$ are generalized cospectral. 
From this, in turn, we infer that the matrix $Q = W_{H}W_{G}^{-1}$ satisfies $QAQ^T= B$ and $Q\1 = \1$. 
Now Lemma \ref{lem:walk} implies that there is a permutation matrix $P$ such that $PW_{G} =  W_{H}$. 
It follows that $Q  = W_{H}W_{G}^{-1} = PW_{G}W_{G}^{-1} = P$,  and hence that $PAP^T = B$. 
Consequently  $G$ and $H$ are isomorphic, and the proof of the theorem is complete. 
\end{proof}

We see therefore that there is a four-fold way to approach the same concept. 

\begin{corollary}
If the graphs $G$ and $H$ are controllable, then the following four conditions are equivalent.
\begin{enumerate}
\item $G$ and $H$ have the same iterated degree sequence.
\item $G$ and $H$ are fractionally isomorphic.
\item $G$ and $H$ are $C^2$-equivalent. 
\item $G$ and $H$ are isomorphic.
\end{enumerate}
\end{corollary}

We  now describe an algorithm that can decide in polynomial time whether two controllable graphs are isomorphic.  
For more details about this procedure we refer the interested reader to the work of Liu and Siemons \cite{liu2022unlocking}. 

Let  $G$ and $H$ be two graphs with adjacency matrices $A$ and $B$, respectively. 
We write $C_G$ and $C_H$ for the companion matrices  of the characteristic polynomials of $G$ and $H$.  
It is easy to verify that $AW_G = W_GC_G$ and $BW_H = W_HC_H$. 
Assuming that $PW_G = W_H$, we find that $(PAP^T-B)W_H = W_H (C_G - C_H)$. 
If $G$ and $H$ are controllable, in particular $W_H$ is invertible, and so $PAP^T - B = W_H (C_G - C_H) W_H^{-1}$.
If, moreover, $G$ and $H$ are cospectral, then $C_G = C_H$. 
Therefore, using the last two assumptions, we have that $PAP^T = B$, and hence that $G$ and $H$ are isomorphic. 

In the process described above we have made three assumptions to conclude that the graphs are isomorphic. 
We observe now that the assumption about cospectrality is superfluous. 
We shall need the following elementary result. 

\begin{lemma}\label{key}
If the sets of vectors $\{\mathbf{u}_1,\dots,\mathbf{u}_m\}$ and $\{\mathbf{v}_1,\dots,\mathbf{v}_m\}$ satisfy $\mathbf{u}_i^T \mathbf{u}_j = \mathbf{v}_i^T \mathbf{v}_j$ for all $i$ and $j$, then there exists an orthogonal matrix $Q$ such that $Q\mathbf{u}_i = \mathbf{v}_i$ for all $i$. 
\end{lemma}

\begin{proof} 
Let us suppose that, for some $k\leq m$, the vectors $\mathbf{u}_1,\dots,\mathbf{u}_k$ form a linearly independent set. 
Take two vectors $\mathbf{u}$ and $\mathbf{v}$ such that $\mathbf{u}^T\mathbf{u} = \mathbf{v}^T\mathbf{v}$ and $\mathbf{u}^T\mathbf{u}_i=\mathbf{v}^T\mathbf{u}_i$ for all $i=1,\dots,k$. 
We show that there is a reflection swapping $\mathbf{u}$ and $\mathbf{v}$ and fixing all the vectors $\mathbf{u}_1,\dots,\mathbf{u}_k$. 
If we write $\mathbf{w} = \mathbf{u}-\mathbf{v}$, then we define the linear transformation $R_\mathbf{w}$ by \[R_\mathbf{w} \mathbf{x} = \mathbf{x} - \frac{2\mathbf{w}^T\mathbf{x}}{\mathbf{w}^T\mathbf{w}}\mathbf{w}\] for all $\mathbf{x}$. 
It is clear that $R_\mathbf{w}\mathbf{w}=-\mathbf{w}$ and $R_\mathbf{w}\mathbf{x} = \mathbf{x}$ for all $\mathbf{x} \in \mathbf{w}^{\bot}$. 
Since $R_\mathbf{w}$ is a reflection, we have that $R_\mathbf{w}$ determines an orthogonal matrix $Q$ such that $Q\mathbf{u} = \mathbf{v}$ and $Q\mathbf{u}_i = \mathbf{u}_i$ for all $i$; induction on $k$ establishes the assertion. 
\end{proof}

An immediate consequence of Lemma \ref{key} is that if $U$ and $V$ are the matrices whose columns are the sets of vectors $\{\mathbf{u}_1,\dots,\mathbf{u}_m\}$ and $\{\mathbf{v}_1,\dots,\mathbf{v}_m\}$ respectively and if $U^TU = V^TV$, then there is an orthogonal matrix $Q$ such that $QU=V$. 
Now, the \emph{extended walk matrices} of $G$ and $H$ are the $n\times (n+1)$ matrices $\widehat{W}_G$ and $\widehat{W}_H$ defined by 
\[
\widehat{W}_G = \big[ \1\quad A\1\quad \cdots \quad A^{n}\1 \big]\qquad\textnormal{and}\qquad\widehat{W}_H = \big[ \1\quad B\1\quad \cdots \quad B^{n}\1 \big].
\]
It follows from Lemma \ref{key} that if $\widehat{W}_G^T\widehat{W}_G = \widehat{W}_H^T\widehat{W}_H$, then there is an orthogonal matrix $Q$ such that $Q\widehat{W}_G = \widehat{W}_H$. 
This implies that $QA^r\1 = B^r\1$ for all $r=0,\dots,n$. 
Since $Q\1 = \1$, then we have  $QA^nQ^T \1 = QA^nQ^TQ\1 =QA^n\1 = B^n\1$. 
From this and from the definition of the companion matrix, we can conclude that the characteristic polynomials of $QAQ^T$ and $B$ have exactly the same coefficients, and hence that $G$ and $H$ are cospectral. 
Finally, if we assume that $P\widehat{W}_G = \widehat{W}_H$ for some permutation matrix $P$, then we see that $\widehat{W}_G^T\widehat{W}_G = \widehat{W}_G^TP^TP\widehat{W}_G = \widehat{W}_H^T\widehat{W}_H$ and also that $PW_G = W_H$. 
Therefore, if two graphs are controllable and if their extended walk matrices are equal up to a permutation of the rows, then the graphs are isomorphic. 
This completes our argument.

%%%%%%%%%%%%%%%%%%%%%%%%%%%%%%%%%%%%%%%%%%%%%%%%%%%%%%%%%%%%%%%%%%%%%%%%%%%%%%%%%%%%%%%%%%%%%%%
\section*{Acknowledgments}
%%%%%%%%%%%%%%%%%%%%%%%%%%%%%%%%%%%%%%%%%%%%%%%%%%%%%%%%%%%%%%%%%%%%%%%%%%%%%%%%%%%%%%%%%%%%%%%
The research of A. Abiad is partially supported by the FWO grant 1285921N. The research of O. Zapata is partially supported by the SNI grant 620178.

%The citation must be used in following style: \cite{article-minimal} \cite{article-full} \cite{article-crossref} \cite{whole-journal}.

\bibliography{main}
\bibliographystyle{plain}


\end{document}